\documentclass{article}[16pt]
\usepackage{amssymb}
\usepackage{amsthm}
\usepackage{amsmath}
\usepackage{tikz}
\usepackage{multicol}
\usepackage{enumitem}

\newtheorem{remark}{Remark}
\newtheorem{lemma}{Lemma}
\newtheorem*{claim}{Claim}
\newtheorem*{problem}{Problem}
\newtheorem{corollary}{Corollary}
\newtheorem{theorem}{Theorem}
\newtheorem*{thm}{Main Theorem}

\begin{document}
\title{There are level ternary circular square-free words of length $n$ for $n\ne 5,7,9,10,14,17.$}
\author{James D. Currie\\
Department of Mathematics \&
Statistics\\
The University of Winnipeg\thanks{The author is
supported by an NSERC Discovery Grant}\\
{\tt currie@uwinnipeg.ca}\vspace{.1in}\\
Jesse T. Johnson\\
Department of Mathematics \&
Statistics\\
University of Victoria\\
{\tt jessejoho@gmail.com}}
\maketitle \abstract{\noindent A word is level if each letter appears in it the same number of times, plus or minus 1. We give a complete characterization of the lengths for which level ternary circular square-free words exist. Key words: combinatorics on words, circular words, necklaces, square-free words, non-repetitive sequences}

\section{Introduction}
Combinatorics on words began with the work of Thue  \cite{thue06}, who showed that there are arbitrarily long square-free words over a three letter alphabet. Ternary square-free words remain an object of study, and progress has been made regarding their enumeration by length \cite{brandenburg,brinkhuis,ekhad,grimm,kolpakov,shur12} 
, the topology of infinite ternary square-free words   \cite{kobayashi,shelton,sheltonsoni}, and  their entropy \cite{baake}. Infinite square-free words do not exist for alphabets with fewer than three letters. One attempt to reduce this necessary alphabet size is to seek an infinite ternary square-free word which is `mostly' binary, i.e., such that  the frequency of one letter as small as possible. It has been shown that a lower bound on the minimum frequency is $883/3215$  \cite{khalyavin,tarannikov}. 

The results mentioned above are for linear words. Thue also studied circular words  \cite{thue12}, and completely characterized the circular overlap-free words on two letters. Surprisingly, it was not until 2002 that the first author \cite{currie} characterized the lengths for which ternary circular square-free words exist: 

\begin{theorem}\label{ternary}
For every positive integer $n$ other than 5, 7, 9, 10, 14, or 17, there is a 
ternary circular square-free word of length $n$.
\end{theorem}

  Several other proofs \cite{shur10,clokie,mol} of this theorem have now been given, signaling increasing interest in circular words. (See also \cite{fitz,mousavi}.)
Circular ternary square-free words are inherently harder to study than linear words; part of this is because the set of linear square-free words is closed under taking factors, while this is not true for the circular case. In the present paper, we consider a sort of opposite problem to that of finding a word with minimal frequency for some letter. Instead, we seek circular square-free words over $\{a,b,c\}$ where letter frequencies are as similar as possible: We seek ternary circular square-free words $w$ such that  for any letters $x,y\in\{a,b,c\}$,
 $$|w|_x-1\le|w|_y\le|w_x|+1.$$
 
 We call such words {\bf level}. We prove the following:

\begin{thm}
There is  a level ternary circular square-free word of length $n$,
for each positive integer $n$, $n\ne 5, 7, 9, 10, 14, 17$. \end{thm}

This result is of interest for its own sake, but also as a tool. The lengths of the images of level words under morphisms are relatively simple to analyze. Thus a level circular square-free word of length $n$ may be useful to obtain a circular word of length $m$ avoiding  some pattern. For example, in  \cite{johnson}, level ternary circular square-free words are used to build binary circular words containing only three square factors, and having specified lengths.

It would be good to obtain similar results for ternary circular square-free words as have been proved for linear words. Thus, the results mentioned above motivate two natural open problems:

\begin{problem} For each positive integer $n$, how many circular square-free words over $\{a,b,c\}$ are there of length $n$?
\end{problem}

\begin{problem} What is the least possible frequency for a letter in a circular square-free word over $\{a,b,c\}$? This can be asked with the length of the word specified, or the limit infimum  can be studied as the length goes to infinity.
\end{problem}

\section{Preliminaries}

For general background on combinatorics on words, see the works of Lothaire   \cite{lothaire83,lothaire02}.
Let $\Sigma$ be a finite set. We refer to $\Sigma$ as an {\bf alphabet}, and its elements as {\bf letters}.  We denote by $\Sigma^*$ the free monoid over $\Sigma$, with identity $\epsilon$, the {\bf empty word}. We call the elements of $\Sigma^*$ {\bf words}. Informally, we think of the elements of $\Sigma^*$ as finite strings of letters, and of its binary operation as concatenation. Thus, if $u=u_1u_2\cdots u_n$, $u_i\in\Sigma$ and $v=v_1v_2\cdots v_m$, $v_j\in\Sigma$, then $uv=u_1u_2\cdots u_nv_1v_2\cdots v_m$. In this case, we say that $u$ is a {\bf prefix} of $uv$ and $v$ is a {\bf suffix}. More generally, if $w=uvz$, then $v$ is a {\bf factor} of $w$.   We call $v$ a {\bf proper factor} of $w$ in the case  $v\ne w$. We say that $v$ appears in $w$ at {\bf index} $i$ in the case where $|u|=i-1$.

Let $u=u_1u_2\cdots u_n$, $u_i\in\Sigma$. We say that $u$ has period $p>0$ if $u_i=u_{i+p}$ whenever $1\le i\le \le n-p$.
A word of the form $s=uu$, $u\ne \epsilon$ is called a {\bf square}. Thus a square $uu$ has period $|u|$. We write $u^2$ for $uu$, $u^3$ for $uuu$, etc. A word $w$ is said to be {\bf square-free} if no factor of $w$ is a square.

We will work in particular with the alphabets  $A=\{a,b,c\}$, $B=\{0,1\}$, and $S=\{1,2,3\}$. Words over $A$ are called {\bf ternary words} and words over $B$ are called {\bf binary words}.

 If $u=u_1u_2\cdots u_n$, $u_i\in\Sigma$, then the {\bf length} of $u$ is defined to be $n$, the number of letters in $u$, and we write $|u|=n$. The set of words of length $m$ over $\Sigma$ is denoted by $\Sigma^m$. We use $\Sigma^{\ge n}$ to denote the set of words over $\Sigma$ of length at least $n$. For $a\in\Sigma$, $u\in\Sigma^*$, we denote by $|u|_a$ the number of occurrences of $a$ in $u$. For $|u|\ge 1$, we use $u^-$ to denote the word obtained by deleting the last letter of $u$; thus $u^-=u_1u_2\cdots u_{n-1}.$ Similarly, if $|u|\ge 2$, then $u^{--}=u=u_1u_2\cdots u_{n-2}.$

If $w=uv$, then define $wv^{-1}=u$. Thus $vu=v(uv)v^{-1}$, and we refer to $vu$ as a {\bf conjugate} of $uv$. The relation `$a$ is a conjugate of $b$' is an equivalence relation on $\Sigma^*$, and we refer to the equivalence classes of $\Sigma^*$ under this equivalence relation as {\bf circular words}. If $w\in\Sigma^*$, we denote the circular word containing $w$ by $[w]$. We may consider the indices $i$ of the letters of a circular word $[u]=[u_1u_2\cdots u_n]$ to belong to ${\mathbb Z}_n$, the integers modulo $n$. Thus $u_{n+1}=u_1$, for example. 
If $[w]$ is a circular word and $v\in\Sigma^*$, we say that $v$ is a {\bf factor} of $[w]$ if $v$ is a factor of an element of $[w]$, i.e., if $v$ is a factor of a conjugate of $w$. A circular word $[w]$ is {\bf square-free} if no factor of $[w]$ is a square.

Let $\Sigma$ and $T$ be alphabets. A map $\mu:\Sigma^*\rightarrow T^*$ is called a {\bf morphism} if it is a monoid homomorphism, that is, if $\mu(uv)=\mu(u)\mu(v)$, for $u,v\in\Sigma^*$.

\section{Outline of proof} Using a result of Shur, we construct ternary circular square-free words via their Pansiot encodings. Our proof takes several steps:
\begin{enumerate}
\item Shur gave conditions on a circular word $[w[$ over $S^*$, such that $[f(w)]$ is the Pansiot encoding of a ternary circular square-free word, where $f(n)=01^n$.
\item\label{words w} We give a morphism $h:A^*\rightarrow S^*$, such that whenever $[v]$ is a circular square-free word over $A$, then $w=h(v)$ satisfies Shur's conditions. The ternary circular square-free word encoded by $w$ is level, and has length 18$|v|$. By Theorem~\ref{ternary}, we therefore can construct level ternary circular square-free words of every length of the form $18n$, $n\ne 5, 7, 9, 10, 14, 17$.
\item\label{words s} We give a condition on words $s$, such that for the words $w$ constructed in step \ref{words w}, $[ws]$ encodes a level ternary circular square-free word.
\item By computer search, we find words $s$ satisfying the condition of step \ref{words s}, having $52\le |s|\le 107.$ This implies that there are level ternary circular square-free words of length $18n+i$, $n\ne 5,7,9,10,14,17$, $52\le i\le 107$. This implies that there is a level
ternary circular square-free word for every length 70 or greater.
\item A final computer search for level
ternary circular square-free words of lengths 69 or less shows that there exists such a word, except for lengths 5,7,9,10,14, and 17.
\end{enumerate}

\section{Shur's conditions}
{\bf Pansiot encodings} were developed to solve Dejean's conjecture \cite{pansiot}. For our purposes, we do not need to consider the general case, and only consider encodings for words over $A=\{a,b,c\}$. Suppose that $v=v_1v_2v_3\cdots v_n$ is a word, $v_i\in A$, $1\le i\le n$, some $n\ge 2$, and $v$ contains no length 2 squares. It follows that given $v_{i}$ and $v_{i+1}$, there are only 2 choices for $v_{i+2}$, either $v_{i+2}=v_{i}$, or  $v_{i+2}$ is the unique element of $A-\{v_i,v_{i+1}\}$. Using this observation, we can encode $v$ by $v_1$, $v_2$, and the binary word $u=u_1u_2\cdots u_{n-2}$, where
\begin{eqnarray}\label{Pansiot}
u_i&=&\left\{
\begin{array}{ll}
0,&v_i=v_{i+2}\\
1,&v_i\ne v_{i+2}
\end{array}\right., 1\le i\le n-2.
\end{eqnarray}
For example, if $v=abcbacbcabcbacb$, then $u=1011101110111$. We call $u$ the {\bf Pansiot encoding} of $v$. By construction, a word can be recovered from its Pansiot encoding if we know its length 2 prefix. 
If two words $z$ and $w$ have the same Pansiot encoding, word $w$  is obtained from $z$ by the permutation of $A$ which maps the length 2 prefix of $z$ to the length 2 prefix of $w$. We call such words {\bf equivalent}. Thus $w=bcacbacabcacbac$ also has Pansiot encoding $u$ above, and is equivalent to $v$; it is obtained from $v$ by the permutation $a\rightarrow b$, $b\rightarrow c$, $c\rightarrow a$.  Write $\pi(v)$ for the Pansiot encoding of a word $v$, and given a binary word $u$, let $\Delta(u)$ be the word with prefix $ab$ and Pansiot encoding $u$. 

Consider the directed graph $D_1$ of Figure~\ref{pansiot pairs}. The vertices of the graph are the length 2 square-free words over $A$. The edges are from $xy$ to $yz$, where $x,y,z\in A$, and each edge is labeled by the Pansiot encoding of $xyz$. By induction, there is a walk from vertex $xy$ to vertex $zw$ labeled by $u$ exactly when there is a word $v$ with prefix $xy$, suffix $zw$, and $\pi(v)=u$. We conclude that the length 2 prefix and length 2 suffix of $v$ are identical if and only if $u=\pi(v)$ labels a closed walk on $D_1$. Note that if $u$ labels a closed walk on $D_1$, then each conjugate of $u$ labels the same closed walk, perhaps starting at a different vertex.

\begin{remark} Pansiot's exposition does not use $D_1$, but instead uses the group which has $D_1$ for its Cayley graph.
\end{remark}

\begin{figure}
\begin{center}
\begin{tikzpicture}[-,shorten >=1pt,auto,node distance=3cm,
  thick,main node/.style={circle,fill=gray!20,draw,font=\sffamily\Large\bfseries},minimum size=1.2cm]
  
  \node[main node] (1) {$ab$};
  \node[main node] (2) [below right of=1] {$bc$};
  \node[main node] (3) [right of=2] {$cb$};
  \node[main node] (4) [above right of=3] {$ba$};
  \node[main node] (5) [below left of=2] {$ca$};
  \node[main node] (6) [below right of=3]{$ac$};

    \draw [->] (1)  edge  node[left] {1} (2)
          edge node[above] {0} (4)
	  (2)  edge  node[above] {0} (3)
          edge node[above] {1} (5)
    (3)  edge  node[right] {1} (4)
        edge node[right] {} (2)
    (4)  edge  node[right] {1} (6)
        edge node[left] {} (1)
    (5)  edge  node[left] {1} (1)
        edge node[above] {0} (6)
    (6)  edge  node[right] {1} (3)
        edge node[] {} (5);
 
\end{tikzpicture}
\end{center}
\caption{\label{pansiot pairs} Graph $D_1$}
\end{figure}
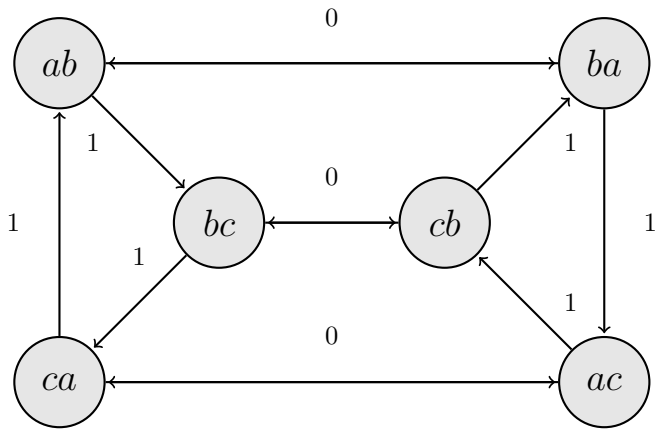

\begin{lemma}\label{Pansiot square}
Let $u\in B^*$. Then $\Delta(u)$ is a square if and only if $u$ can be written in the form $u=V\chi\upsilon V$, some $V\in B^*, \chi,\upsilon\in B$,   such that $V\chi\upsilon$ labels a closed walk on $D_1.$
\end{lemma}
\begin{proof} Suppose that $\Delta(u)=zz$, some $z\in A^+$. Since $zz$ has period $|z|$, so does $u=\pi(zz)$. Since $ab$ is a prefix of $z$, word $zz$ has prefix $zab$, which begins and ends with $ab$. It follows that prefix $q=\pi(zab)$ of $u$ labels a closed walk on $D_1$. Letting $V=\pi(z)$, we see that $u$ has the desired form.

In the other direction, suppose that $u$ can be written in the form $u=V\chi\upsilon V$, some $V\in B^*, x,y\in B$,   such that $V\chi\upsilon$ labels a closed walk on $D_1.$ The prefix and suffix of $\Delta(V\chi\upsilon V)$ which have length $|\Delta(V)|$ are equivalent words, since each has Pansiot encoding $V$. They will be identical if $\Delta(V\chi\upsilon)$ ends in $ab$. However, $V\chi\upsilon$ labels a closed walk in $D_1$, so that $\Delta(V\chi\upsilon)$ indeed ends in $ab$.
\end{proof}

\begin{remark}\label{conjugates}Suppose that $\Delta(u)=zz$. Since $u$ has period $|z|=|V\chi\upsilon|$, all length $|V\chi\upsilon|$ factors of $u$ are conjugates of $V\chi\upsilon$, and therefore also label closed walks.\end{remark}
 
For a word $v\in A^*$ to have a Pansiot encoding, $v$ can have no length 2 squares. Suppose that $v$ contains no squares of any length. This implies that the Pansiot encoding $u=\pi(v)$ obeys certain restrictions. For example, 00 cannot be a factor of $u$, or else $v$ contains a factor equivalent $\Delta(00)=abab$, which is a square. Similarly, 1111 cannot be a factor of $u$ since $\Delta(1111)=abcabc$ is a square.  It therefore follows that any factor of $u$ of the form $0w0$ can be written as $0w0=f(U)0$, $U\in S^*$, where
\begin{eqnarray*}
f(1)&=&01\\
f(2)&=&011\\
f(3)&=&0111.
\end{eqnarray*}

\begin{lemma}\label{U factor} Suppose that $U\in S^*$ has one of 
$11$, $222$, $223$, $322$, or $333$ as a proper factor.
Then $\Delta(f(U))$ contains a square. 
\end{lemma}
\begin{proof}One checks that each right or left extension of these words leads to a square in $\Delta(f(U))$. For example, if 222 is a proper factor of $U$, then  $U$ contains a word of one of the forms $x222$ and $222x$, where $x\in S$. In the first case, $f(U)$ contains a factor $1011011011$,  so that  $\Delta(f(U))$ contains a factor equivalent to $\Delta(1011011011)=abcbacabcbac$, which is a square; in the second case, $f(U)$ contains 0110110110, and $\Delta(f(U))$ contains a factor equivalent to the square $\Delta(0110110110)=abacbcabacbc$.
\end{proof}

Moving to the level of the Pansiot encoding, we therefore have the following:
\begin{lemma} If $u$ is the Pansiot encoding of a square-free word over $\{a,b,c\}$, then $u$ is a factor of a word $f(U)$, some word $U\in S^*$, such that $U$ does not contain a proper factor $11$, $222$, $223$, $322$ or $333$.
\end{lemma}

\begin{remark}\label{D_2} Consider the directed graph $D_2$ of Figure~\ref{U-walks}. The vertices of the graph are again the length 2 square-free words over $A$, as in $D_1$. For each $\alpha\in S$ and for each vertex $xy$, there is an edge from $xy$ to $zw$, labeled by $\alpha$, exactly when $zw$ is the endpoint of the walk in $D_1$ labeled by $f(\alpha)$ starting at $xy$. Thus $U\in S^*$ labels a closed walk on $D_2$, exactly when $f(U)$ labels a closed walk on $D_1$. 

Note that $D_2$ is bipartite, with bipartition $V_1=\{ab,bc,ca\}$ and $V_2=\{ba,cb,ac\}$. Note also, that $D_2$ is highly symmetric, so that if there is a closed walk labeled by $U$ starting at one of the vertices, there is also a closed walk labeled by $U$ starting at each of the other vertices. 
\end{remark}

	Let $s=s_1s_2\cdots s_n$, $s_i\in  S$, $1\le i\le n$. Let $\omega(s)=\sum\limits_{i=0}^{|s|} (-1)^{i-1} s_i$, with each letter $s_i$ of $s$ considered as an integer. 

\begin{lemma}\label{omega}
	Let $s\in S^*$. The following are equivalent:
	\begin{enumerate}
	\item Length $|s|$ is even, and $\omega(s)\equiv 0$ (mod 3).
	\item Word $s$ is the sequence of edge labels of a closed walk on $D_2$.
\end{enumerate}
\end{lemma}
\begin{proof} Consider the function $g$ on the vertices of $D_2$ where $g(ab)=g(ba)=0$, $g(ca)=g(ac)=1$, $g(bc)=g(cb)=2$. One checks that if $x\in\{ab,bc,ca\}$, then if there is an edge from $x$ to $y$ labeled $z$, we have $g(y)\equiv g(x)+z$ (mod 3). On the other hand, 
if $x\in\{ba,cb,ac\}$, then if there is an edge from $x$ to $y$ labeled $z$, we have $g(y)\equiv g(x)-z$ (mod 3). By induction, we get the following:
\begin{claim}
If $s$ labels a walk starting at $x\in\{ab,bc,ca\}$, then the walk ends at a vertex $y$ with $g(y)\equiv g(x)+\omega(s)$ (mod 3).
If $s$ labels a walk starting at $x\in\{ba,cb,ac\}$, then the walk ends at a vertex $y$ with $g(y)\equiv g(x) -\omega(s)$ (mod 3).
\end{claim}
 Suppose $|s|$ is even, and $\omega(s)\equiv 0$ (mod 3). If $s$ labels a walk starting at $x\in\{ab,bc,ca\}$, then since $D_2$ is bipartite and $|s|$ is even,
the walk ends at a vertex $y\in\{ab,bc,ca\}.$ By the claim,
$g(y)\equiv g(x)+\omega(s)\equiv g(x)$ (mod 3).
 Since $g$ is 1-1 on $\{ab,bc,ca\}$, hence $x=y$ and the walk is closed. The proof is
  similar if $s$ labels a walk starting at $x\in\{ba,cb,ac\}$.
  
Suppose that $s$ is the sequence of edge labels of a closed walk on $D_2$. Since $D_2$ is bipartite, $|s|$ is even. Choose a vertex $x=y\in\{ab,bc,ca\}$ as the beginning and end of the walk. By the claim, $g(x)=g(y)\equiv g(x)+\omega(s)$ (mod 3). We conclude that $\omega(s)\equiv 0$ (mod 3).
  \end{proof}

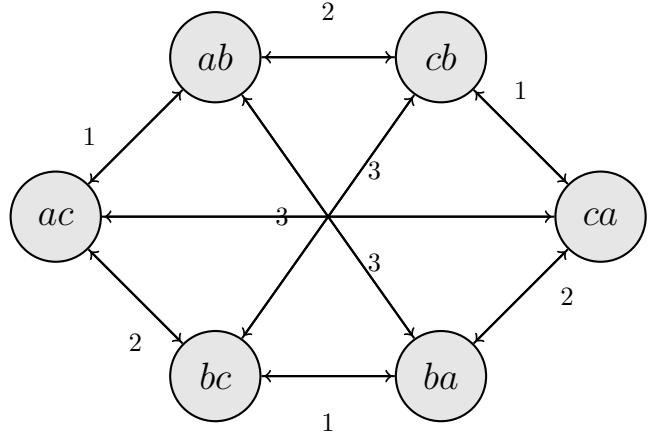
\begin{figure}
\begin{center}
\begin{tikzpicture}[-,shorten >=1pt,auto,node distance=3cm,
  thick,main node/.style={circle,fill=gray!20,draw,font=\sffamily\Large\bfseries},minimum size=1.2cm]
  
  \node[main node] (1) {$ab$};
  \node[main node] (2) [right of=1] {$cb$};
  \node[main node] (3) [below right of=2] {$ca$};
  \node[main node] (4) [below left of=3] {$ba$};
  \node[main node] (5) [left of=4] {$bc$};
  \node[main node] (6) [below left of=1]{$ac$};

    \draw [->] (1)  edge  node[above right] {3} (4)
          edge node[above] {2} (2)
          edge node[left] {1}(6)
	  (2)  edge  node[below right] {3} (5)
          edge node[] {} (3)
          edge node[] {} (1)
	  (3)  edge  node[left] {3} (6)
          edge node[right] {2} (4)
          edge node[above] {1} (2)
	  (4)  edge  node[] {} (1)
          edge node[] {} (5)
          edge node[] {} (3)
	  (5)  edge  node[right] {} (2)
          edge node[below] {2} (6)
          edge node[below] {1} (4)
	  (6)  edge  node[left] {} (1)
          edge node[above] {} (3)
          edge node[above] {} (5);
 
\end{tikzpicture}
\end{center}
\caption{\label{U-walks} Graph $D_2$}
\end{figure}

\begin{lemma}\label{pull back to U} Suppose that $U\in S^*$. Suppose that $\Delta(f(U))$ contains a square. Then  $U$ contains a proper factor $11$, $222$, $223$, $322$, or $333$, or $U$ contains a factor $WxyW$,
 some $W\in S^{\ge 2}, x,y\in S$,   such that $Wxy$ labels a closed walk on $D_2.$
\end{lemma}
\begin{proof}
Suppose that $U$ doesn't contain a proper factor 11, 222, 223, 322, or 333. We wish to show that $U$ contains a factor $Wxy W$,
 some $W\in S^{\ge 2}, x,y\in S$,   such that $Wxy$ labels a closed walk on $D_2.$ 
 
 By Lemma~\ref{Pansiot square}, $f(U)$ contains a factor $u=V\chi\upsilon V$, some $V\in B^*, x,y\in B$,   such that $V\chi\upsilon$ labels a closed walk on $D_1.$ Since $V\chi\upsilon V$ is a factor of $f(U)$, neither of 00 and 1111 is a factor of $V\chi\upsilon V$.
 
Consider words of $B^*$, of length at most 8, labeling closed walks in $D_1$ and not containing a factor 00 or 1111. These are found to be 111, 0101, 011011, 01010101, 01110111  and their conjugates, and 01110. If $|V\chi\upsilon|\le 8$, then it must be one of these words.

\subsubsection*{Case 1: $|V\chi\upsilon|\le 8$.} 

\subsubsection*{Case 1a: A conjugate of $V\chi\upsilon$ is 111.}
 Here $V\chi\upsilon=111$, so that $V\chi\upsilon V=1111$, which is impossible. 
\subsubsection*{Case 1b: A conjugate of $V\chi\upsilon$ is 0101, or 01010101.}
If $V\chi\upsilon$ is a conjugate of 0101, then $V\chi\upsilon V$ is 010101 or 101010, and the factor 01010 in $f(U)$ implies that $U$ has proper factor 11, which is impossible. A factor 01010 in $f(U)$ also arises if $V\chi\upsilon$ is a conjugate of 01010101.
\subsubsection*{Case 1c: A conjugate of $V\chi\upsilon$ is 011011.} In this case, $V\chi\upsilon V$ is one of 0110110110, 1101101101, or 1011011011, forcing $U$ to contain one of 222, 322, or 223 as a proper factor, which is impossible.
\subsubsection*{Case 1d: A conjugate of $V\chi\upsilon$ is 01110111.} In this case, $V\chi\upsilon V$ is one of 01110111011101, 111011101110, 11011101110111, or 10111011101110, forcing $U$ to contain 333 as a proper factor, which is impossible.

\subsubsection*{Case 2: $|V\chi\upsilon|\ge 9.$} 

If $|V\chi\upsilon|\ge 9$, then $|V|\ge 7$. We claim that $|V|_0\ge 2$; otherwise, $|V|_0\le 1$. Since 1111 is not a factor of $f(U)$, this forces $V= 1110111$. However, now $Vx$ is a factor of $f(U)$, forcing $x=0$, and $yV$ is a factor of $f(U)$, forcing $y=0$. But then $xy=00$ is a factor of $f(U)$, which is impossible.

Since $|V|_0\ge 2$, write $V=1^rf(Z)01^s$, for non-negative integers $r$ and $s$, and some $Z\in S^+$. 
We have $V\chi\upsilon V=1^rf(Z)01^s\chi\upsilon 1^rf(Z)01^s$. Since $f(Z)01^s\chi\upsilon 1^rf(Z)0$ is a factor of $f(U)$, $U$ has factor $Z\alpha Z$, where $f(\alpha)=01^s\chi\upsilon 1^r$.  Recall that,  $V\chi\upsilon=1^rf(Z)01^s\chi\upsilon$ labels a closed walk on $D_1$.  It follows that its conjugate $f(Z)01^s\chi\upsilon1^r=f(Z\alpha)$ also labels a closed walk on $D_1$.
By Remark~\ref{D_2}, $Z\alpha$ labels a closed walk on $D_2$. 

At most one of $\chi$ and $\upsilon$ can be 0, since 00 is not a factor of $f(U)$. It follows that $1\le |\alpha|\le 2$. To summarize thus far: Word
$U$ has factor $Z\alpha Z$, word $Z\alpha$ labels a closed walk on $D_2$, and $|\alpha|\le 2$. If $|\alpha|= 2$, and $|Z|\ge 2$, let $Z=W$, $\alpha=xy$, and we are done.

If $|Z|\ge 3$, let $xy$ be the length 2 suffix of $Z\alpha$ and write $Z\alpha=Wxy$. Since $|\alpha|\le 2$, we see that $W$ is a prefix of $Z$, and $WxyW$ is a prefix of $Z\alpha Z$. Since $|W|=|Z|+|\alpha|-2\ge 3+1-2=2$, we are done.

Suppose then that $|Z|\le 2$. Since $Z\alpha$ is a closed walk, we deduce that $|Z\alpha|$ must be even, and $|Z\alpha|\le 4$. If $|Z\alpha|=4$, let $Wxy=Z\alpha$, and we are done, as in the case $|Z|\ge 3$. Therefore, 
suppose $|Z\alpha|=2$. We conclude that $Z\alpha$ must be one of the length 2 closed walks on $D_2$, namely 11, 22 and 33. Since $Z\in S^+$ and $1\le |\alpha|$, we have $|Z|=|\alpha|=1$ and $Z\alpha Z$ is one of 111, 222, 333, none of which is a factor of $U$. This is a contradiction.
\end{proof}

Suppose that $v=v_1v_2v_3\cdots v_n$ is a word, $v_i\in A$, $1\le i\le n$, some $n\ge 2$, and $[v]$ contains no length 2 squares.  Then $v_n\ne v_1$, and $v_1$ is either $v_{n-1}$ or the unique element of $A-\{v_{n-1}, v_n\}$.
Similarly, $v_2$ is either $v_{n}$ or the unique element of $A-\{v_n,v_1\}$. We may thus extend the notion of Pansiot encoding to the circular word $[v]$. The {\bf circular Pansiot encoding} of $[v]$ is the circular binary word $[u]$ where 

\begin{eqnarray}\label{circular Pansiot}
u_i&=&\left\{
\begin{array}{ll}
0,&v_i=v_{i+2}\\
1,&v_i\ne v_{i+2}
\end{array}\right., 1\le i\le n,
\end{eqnarray}

\noindent performing the arithmetic on the indices $i$ modulo $n$. For example, the encoding of $[abcacb]$ is $[110110]$. Here we are using $v_1v_2\cdots v_n=abcacb$, so that (\ref{circular Pansiot}) gives $u_1u_2\cdots u_n=110110$. Moving two letters from the end of $abcacb$ to the beginning gives a different representative of $[abcacb]$, with $v_1v_2\cdots v_n=cbabca$. In this case, (\ref{circular Pansiot}) gives $u_1u_2\cdots u_n=101101$, which is obtained by moving  two letters from the end of $110110$ to the beginning,  and is another representative of $[110110]$. Note that $\Delta(101101)=
abcbacab$. The length 6 prefix of this word is $abcbac$, which is equivalent to $cbabca$.
\begin{remark}\label{closed Pansiot}
If $u$ is a representative of a circular Pansiot word, then $\Delta(u)$ ends in $ab$, so that $u$ labels a closed walk on $D_1$. Any circular word with Pansiot encoding $u$ has the form $[v]$, where $v$ is equivalent to $\Delta(u)^{--}$. Each conjugate $v'$ of $v$ is equivalent to a word $\Delta(u')^{--}$, some conjugate $u'$ of $u$.
\end{remark}
\begin{theorem}
\label{circular U} \cite{shur10} 
Suppose that $U\in S^*$, and $U$ labels a closed walk on $D_2$. Suppose that $[U]$ contains no factor $11$, $222$, $223$, $322$, or $333$, and no factor $WxyW$, such that $W\in S^{\ge 2}, x,y\in S$, and $Wxy$ labels a closed walk on $D_2.$ Then $[f(U)]$ is the circular Pansiot encoding of a circular square-free word.
\end{theorem}
\begin{remark} We refer to the conditions on $U$ in this theorem as {\bf Shur's conditions}.
\end{remark}
\begin{proof} Suppose not. Write $U=U_1U_2\dots U_n$. The result certainly holds in the case $U=\epsilon$, so assume $n>0$. Since $U$ labels a closed walk in a bipartite graph, $n$ must be even, so that $n\ge 2$. Let $u $ be a conjugate of $f(U)$ such that
$\Delta(u)^{--}$ contains a square. Write $u=a_i''a_{i+1}a_{i+2}\cdots a_na_1\cdots a_{i-1}a_i'$ for some $i$, where $a_i=f(U_i)$, $1\le i\le n$, and $a_i=a_i'a_i''$, $a_i'\ne\epsilon$. 

Word $u$ is a factor of $f(U_iU_{i+1}\cdots U_{i-1}U_i)$. By Lemma~\ref{pull back to U}, $U_iU_{i+1}U_{i+2}$ $\cdots$ $ U_nU_1$ $\cdots$ $ U_{i-1}U_i$  has a proper factor 11, 222, 223, 322, or 333, or a factor $WxyW$, $W\in S^{\ge 2}, x,y\in S$, such that $Wxy$ labels a closed walk on $D_2.$ However, if $U_iU_{i+1}U_{i+2}$ $\cdots$ $ U_nU_1$ $\cdots$ $ U_{i-1}U_i$  has a proper factor 11, 222, 223, 322, or 333, then one of the factors $U_p=U_iU_{i+1}U_{i+2}$$\cdots$$ U_nU_1$$\cdots$$ U_{i-1}$ and $U_s=U_{i+1}U_{i+2}$$\cdots$$ U_nU_1$$\cdots$$ U_{i-1}U_i$ of $[U]$
contains a factor 11, 222, 223, 322, or 333, contrary to assumption. 

Again, because
$U_p$ and $U_s$ are factors of $[U]$, neither  
contains a factor $WxyW$, $W\in S^{\ge 2}, x,y\in S$, such that $Wxy$ labels a closed walk on $D_2.$ It follows that we can write
$$ U_iU_{i+1}U_{i+2}\cdots U_nU_1\cdots  U_{i-1}U_i=WxyW,$$ $W\in S^{\ge 2}, x,y\in S$, such that $Wxy$ labels a closed walk on $D_2.$

Recall that $|U|$ is even. Thus $|U_iU_{i+1}U_{i+2}\cdots U_nU_1\cdots  U_{i-1}U_i|=|U|+1$ is odd. However, $|WxyW|=2|W|+2$ is even. This is a contradiction.
\end{proof}
\section{The morphism $h$}

Consider the substitution $h:A^*\rightarrow S^*$ given by
\begin{eqnarray*}
h(a)&=&123123\\
h(b)&=&132132\\
h(c)&=&131313.
\end{eqnarray*}

For $x\in A$ we refer to the word $h(x)$ as a {\bf block}.

\begin{remark}\label{closed walk}
One checks that each block labels a closed walk on $D_2$. Thus, for any word $w\in A^*$, $h(w)$ also labels a closed walk on $D_2$. Also note that none of 11, 22, 33 is a factor of the circular word $[h(w)].$
\end{remark}

Call a non-empty word $q\in S^*$ {\bf ambiguous} if there exist square-free words $\alpha,\beta\in A^*$, and words $p_1,p_2,s_1,s_2\in S^*$, such  that 
$h(\alpha)=p_1q s_1$, $h(\beta)=p_2q s_2\in h(S^*)$, and $|p_1|\not\equiv
|p_2|$ (mod 6). An example of an ambiguous word is 1231, which is verified by letting $\alpha=a$, $\beta=ab$, $p_1=\epsilon$, $s_1=213$, $p_2=123$, $s_2=32132$. 

\begin{lemma}\label{ambig}
The only ambiguous words of even length are 12, 23, 31, 13, 32, 21, 1231, 3123, 1321, 3213, 3131, 1313, 313131, 131313. 
\end{lemma}
\begin{proof}
Since blocks have length 6, any ambiguous word of  length 8 is a factor of $h(\alpha)$, some square-free $\alpha\in A^*$, $|\alpha|\le 3$. An exhaustive search (manual or by computer) shows that there are no ambiguous words of length 8.
From the definition, any factor of an ambiguous word is ambiguous, so that there are no ambiguous words of length 8 or more.

Any ambiguous word of length 6 or less is a factor of $h(\alpha)$, some square-free $\alpha\in A^*$, $|\alpha|\le 2$. An exhaustive search produces the given list.
\end{proof}
\begin{lemma}\label{synchronization}
	The morphism $h$ has the following properties:
	\begin{enumerate}
		\item Let $x,y\in A$. Let $s\in S^2$. If $s$ is a suffix of both of $h(x)$ and $h(y)$, then $x=y$. Thus each letter $x\in A$ is determined by the length 2 suffix of $h(x)$.
		\item Let $x,y\in A$. Let $p\in S^3$. If $p$ is a prefix of both of $h(x)$ and $h(y)$, then $x=y$. Thus each letter $x\in A$ is determined by the length 3 prefix of $h(x)$.
	\end{enumerate}
\end{lemma}
\begin{proof}
	The proof is by inspection:
	\begin{enumerate}
		\item The length 2 suffix of $h(a)$ is 23, the length 2 suffix of $h(b)$ is 32, and the length 2 suffix of $h(c)$ is 13. Therefore, the length 2 suffixes of blocks are distinct.
		\item The length 3 prefix of $h(a)$ is 123, the length 3 prefix of $h(b)$ is 132, and the length 3 prefix of $h(c)$ is 131. Therefore, the length 3 prefixes of blocks are distinct.
	\end{enumerate}
\end{proof}

\begin{theorem}\label{basicSquareFree}
	Let $[v]$ be a circular square-free word over $A$. Let $w= h(v)$. Then $f(w)$ encodes a circular square-free word. 
\end{theorem}
\begin{proof} 
One checks the result when $|v|=1$. Suppose then that $|v |\ge 2$.
It suffices to show that $[w]$ satisfies Shur's conditions. Certainly $w$ labels a closed walk, and 
no element of $\{11,222,223,322,333\}$ can appear as a factor of $[w]$ by Remark~\ref{closed walk}.	
	Suppose for the purpose of finding a contradiction that some conjugate $\hat{w}$ of $w$ contains a factor $VxyV$, where $x,y \in S$, $V\in S^{\ge 2}$, and $Vxy$ labels a closed walk on $D_2$.  
Write $v=v_1v_2\cdots v_n$, $v_i\in A$, $1\le i\le n$. Replacing $v$ by one of its conjugates if necessary, we suppose that $\hat{w}$ is a factor of $h(vv_1)$.

\noindent{\bf Claim.}
Word $vv_1$ is square-free.

\noindent{\bf Proof of claim.} The length $n$ prefix and suffix of $vv_1$ are square-free, since they are  factors of $[v]$. Thus if $vv_1$ contains a square $zz$, we must have $vv_1=zz$. However, then, since $v_1$ is both a prefix and suffix of $z$, we see that $vv_1$ contains the length 2 square $v_1v_1$ at index $|z|$. Since $n\ge 2$, this gives a square in $v$, which is impossible.

\noindent{\bf End proof of claim.}
\subsubsection*{Case 1: Word $V$ is ambiguous.}

Since $Vxy$ labels a closed walk on $D_2$, $|V|$ is even. It follows that $V$ is one of the words listed in Lemma~\ref{ambig}. Using Lemma~\ref{omega}, we find 
\begin{eqnarray*}
0&\equiv& \omega(Vxy)\\
&\equiv& \omega(V)+\omega(xy) \mbox{ (mod 3)},
\end{eqnarray*}
implying 	$\omega(xy)\equiv -\omega(V)$ (mod 3). Combined with the fact that $Vx$ and $yV$ must be factors of $h(\alpha)$ for some $\alpha$, we rapidly show each $V$ listed in Lemma~\ref{ambig} gives a contradiction, by considering the possibilities for $xy$. 

For example, suppose $V=1321$. In this case $\omega(V)=-1$, so we require $\omega(xy)\equiv 1$. The possible values for $xy$ with $\omega(xy)\equiv 1$ (mod 3) are $xy=$ 21, 32, or 13. Since $x=1$ implies 11 is a suffix of $Vx$, and $y=1$ implies 11 is a prefix of $yV$, the only remaining possibility is $xy=32$. Now, however, $VxyV=1321321321$, and 132132132 is not a factor of $h(\alpha)$ for any square-free $\alpha\in A^*$. We conclude that $V=1321$ is impossible.

In this way, all the possibilities can be ruled out. Alternatively, since we have $|VxyV|\le 14$ in each case, $VxyV$ would have to be a factor of $h(\alpha)$ for some square-free $\alpha$ of length at most 4. In addition, we require $\omega(Vxy)\equiv 0$ (mod 3). Each possibility for $V$, $x$, and $y$ can thus be ruled out by a computer search examining $h(\alpha)$ for each square-free $\alpha\in A^4$.
\subsubsection*{Case 2: Word $V$ is not ambiguous.}
Since $VxyV$ is a factor of $\hat{w}$, which is a factor of $h(vv_1)$, we see $V$ appears in $h(vv_1)$ at indices which differ by $|Vxy|$. Since $vv_1$ is square-free and $V$ is a non-empty word which is not ambiguous, it follows that $|Vxy|\equiv 0$ (mod 6).

Thus $h(vv_1)$ contains a factor $u$ with period $q$, $q\equiv 0$ (mod 6), $|u|=2q-2$. Let $u_1u_2\cdots u_m$ be a shortest factor of $vv_1$ such that $h(u_1u_2\cdots u_m)$ contains $u$. Write $h(u_i)=U_i$, $1\le i\le m$, and
$u=U_1''U_2\cdots U_{m-1}U_m'$, where $U_1=U'_1U_1''$, $U_m=U'_mU_m''$. Since $u_1u_2\cdots u_m$  is as short as possible, $U_1'',U_m'\ne\epsilon$. 

Let the prefix of $u$ of length $q$ be
$u_p=U_1''U_2\cdots U_j'$, where $U_j=U'_jU_j''$, $U_j'\ne\epsilon$.  
Since each $U_i$ is a block, of length 6, we have 
\begin{eqnarray*}
|U_j'U_j''|&\equiv&0\\
&\equiv&q\\
&\equiv&|U_1''U_2\cdots U_j'|\\
&\equiv&|U_1''U_j'|\mbox{ (mod 6)},
\end{eqnarray*}
so that $|U_1''|\equiv|U_j''|$ (mod 6). 

\subsubsection*{Case 2a: $|U_1''|=6$.} If $|U_1''|=6$, since $U_j'\ne\epsilon$, we have $|U_j''|=0$, and $u_p=U_1\cdots U_j$, $u=U_1\cdots U_jU_{j+1}\cdots U_{m-1}U_m'$. The length of $u$ is $2q-2\equiv 4$ (mod 6), forcing $|U_m'|=4$. Since $u$ has period $q=|U_1\cdots U_j|$, we find
\begin{eqnarray*}
U_i&=&U_{j+i}, 1\le i\le j-1\\
m&=&2j\\
\hat{U}_j&=&U_m',
\end{eqnarray*}
where $\hat{U}_j$ is the length 4 prefix of $U_j$. By Lemma~\ref{synchronization}, we find that $u_i=u_{j+i}, 1\le i\le j$, and $u_1\cdots u_m=(u_1\cdots u_j)^2$ is a square factor of $vv_1$. This is a contradiction.
\subsubsection*{Case 2b: $|U_1''|= 5$.}
In this case, $|U_1''|\equiv|U_j''|$ forces $|U_1''|=|U_j''|$. We therefore find $|U_j'|=1$. From $|u|=2q-2$, we find $|U_m'|=5$. The fact that $u$ has period $q$ forces
\begin{eqnarray*}
U_1''&=&U_j''\\
U_i&=&U_{j+i-1}, 2\le i\le j-2\\
m&=&2j-2\\
U'_{j-1}&=&U_m',
\end{eqnarray*}
where $U'_{j-1}$ is the prefix of $U_{j-1}$ of length 5.
By Lemma~\ref{synchronization}, we find that $u_i=u_{j+i-1}, 1\le i\le j-1$, and $u_1\cdots u_m=(u_1\cdots u_{j-1})^2$ is a square factor of $vv_1$. This is a contradiction.
\subsubsection*{Case 2c: $2\le|U_1''|\le 4$.}
In this case, $|U_1''|=|U_j''|$. The fact that $u$ has period $q$ forces
\begin{eqnarray*}
U_1''&=&U_j''\\
U_i&=&U_{j+i-1}, 2\le i\le j-1\\
m&=&2j-1.
\end{eqnarray*}
 By Lemma~\ref{synchronization}, we find that $u_i=u_{j+i-1}, 1\le i\le j-1$, and $u_1\cdots u_{m-1}=(u_1\cdots u_{j-1})^2$ is a square factor of $vv_1$. This is a contradiction.
\subsubsection*{Case 2d: $|U_1''|=1$.}
In this case, $|U_j'|=5$. The fact that $u$ has period $q$ forces
\begin{eqnarray*}
U_i&=&U_{j+i-1}, 2\le i\le j-1\\
m&=&2j-1\\
\hat{U}_j&=&U_m',
\end{eqnarray*}
where $\hat{U}_j$ is the length 3 prefix of $U_j$. By  Lemma~\ref{synchronization}, we find that $u_i=u_{j+i-1}, 2\le i\le j$, and $u_2\cdots u_m=(u_2\cdots u_j)^2$ is a square factor of $vv_1$. This is a contradiction.
\end{proof}

We wish to show that if $v\in A^*$, then $h(v)$ encodes a level word. 
\begin{remark}\label{concatenation}
Let $\mu,\nu\in B^*$. If $\Delta(\mu)$ ends in $ab$, then $\Delta(\mu\nu)=\Delta(\mu)^{--}\Delta(\nu).$ For each block $h(x)$, $\Delta(f(h(x)))$ ends in $ab$. (See Table~\ref{block decodings}.) It follows that if $a_i\in A$, $1\le i\le n$, then
\begin{eqnarray*}
&&\Delta(f(h(a_1a_2\cdots a_{n-1}a_n)))\\
&=&\Delta(f(h(a_1))f(h(a_2))\cdots f(h(a_{n-1}))f(h(a_n)))\\
&=&\Delta(f(h(a_1)))^{--} \Delta(f(h(a_2)))^{--}\cdots \Delta(f(h(a_{n-1})))^{--}\Delta(f(h(a_n)))
\end{eqnarray*}  
which corresponds to the circular word $$\Delta(f(h(a_1)))^{--} \Delta(f(h(a_2)))^{--}\cdots \Delta(f(h(a_{n-1})))^{--}\Delta(f(h(a_n)))^{--}.$$ From Table~\ref{block decodings}, we see that for each block $h(x)$, $\Delta(f(h(x)))^{--}$ contains each letter $a,b,c$ exactly 6 times.\end{remark} 

\begin{corollary}\label{basic_squarefree_level}
	Let $[v]$ be a circular square-free word over alphabet $A$. Let $w= h(v)$. Then $f(w)$ encodes a level ternary circular square-free word. In fact, $|\Delta(f(w))|_a=|\Delta(f(w))|_b=|\Delta(f(w))|_c$.
\end{corollary}
	
	\begin{table}[h!]
    \centering
    	\begin{tabular}{|c|c|c|}
    	    \hline
    	    $x$ & $h(x)$ & $\Delta(f(h(x)))$\\
    		\hline
    		$a$ & 123123 & abacabcbacbcacbabcab\\
    		$b$ & 132132 & abacabcacbabcbacbcab\\
    		$c$ & 131313 & abacabcacbcabcbabcab\\
    		\hline
    	\end{tabular}\\
    	\caption{The decodings of the images of blocks under $f$}
    	\label{block decodings}
    \end{table} 

Theorem~\ref{basicSquareFree} may be used to construct a square-free level ternary word of any length of the form $18n$, with $n\ne 5,7,9,10,14,17.$
\section{The words $s$}

\begin{theorem}\label{linking words}
	Let $v\in A^*$ be a word with prefix $a$ and suffix $b$, such that $[v]$ is a circular square-free word. Let $w= h(v)$. Suppose $s=33T22\in S^*$, and:
	
	\begin{enumerate}
		\item The word $s$ labels a closed walk on $D_2$;
	\item\label{no h}  Word $s$ has no factor $h(\mu)$, where $\mu\in A^{\ge 2}$, and $[\mu]$ 		is square-free;
		\item \label{Vxy} Word $2s1$ contains no factor $VxyV$,  $V\in S^{\ge 2}$, $x,y\in S$
		where $Vxy$ labels a closed walk on $D_2$;
		\item \label{1's} The word $T$ begins and ends with the letter 1;
		\item \label{length 2} The word $T$ contains no length 2 square;
		\item\label{prefix} Word $T$ has no prefix of the form $qh(u)13213$ or $1qh(u)13213$, where $u\in A^*$ and $q$ is a suffix of a block;
\item\label{suffix} Word $T$ has no suffix of the form $23123h(u)p$ or $23123h(u)p1$, where $u\in A^*$ and $p$ is a prefix of a block.
	\end{enumerate}
	Then $f(ws)$ encodes a circular square-free word.
\end{theorem}

\begin{proof}
	
	We show that $[ws]$ satisfies Shur's conditions. Both $w$ and $s$ label closed walks on $D_2$, so that $ws$ does also. By conditions \ref{1's} and \ref{length 2}, $[ws]$ has no factor from $\{11, 222, 223, 322, 333\}$. Suppose then, for the sake of getting a contradiction, that $[ws]$ contains a factor of the form $VxyV$ where $V\in S^{\ge 2}$, $x,y\in S$, and $Vxy$ labels  a closed walk on $D_2$.
	
By Theorem~\ref{basicSquareFree} and condition~\ref{Vxy}, $Vxy$ is not a factor of $w$ or of $2s1$.

 There are four possibilities:
	
	\begin{enumerate}[wide, labelwidth=!, labelindent=0pt]
		\item[Case 1: ] $VxyV=s''w'$, where $s''$ is a non-empty suffix of $s$, and $w'$ is a non-empty prefix of $w$.
		\item[Case 2: ] $VxyV=s''ws'$, where $s''$ is a non-empty suffix of $s$, and $s'$ is a non-empty prefix of $s$.
		\item[Case 3: ] $VxyV=w''s'$, where $w''$ is a non-empty suffix of $w$, and $s'$ is a non-empty prefix of $s$.
		\item[Case 4: ] $VxyV=w''sw'$, where $w''$ is a non-empty suffix of $w$, and $w'$ is a non-empty prefix of $w$.
	\end{enumerate}
	
	\begin{remark}\label{repetition_restriction}
		Note that $33$ appears in $[ws]$ only once, as a prefix of $s$. Similarly, $22$ appears in $[ws]$ only as a suffix of $s$. Therefore, in any of these cases, it is impossible for a length two prefix or suffix of $s$ to appear entirely inside $V$. Any non-empty prefix or non-empty suffix of $s$ appearing in $VxyV$ must either have length 1, or else one of its repeated letters must occur at $x$ or $y$.
	\end{remark}

\subsubsection*{Case 1: $VxyV=s''w'$, where $s''$ is a non-empty suffix of $s$, and $w'$ is a non-empty prefix of $w$.} 
Because of the restriction mentioned in Remark \ref{repetition_restriction} the possible subcases are:

	\begin{enumerate}[wide, labelwidth=!, labelindent=0pt]
	    \item[Case 1a:] $|s''|=1$.
	    \item[Case 1b:] $|s''|=|V|+1$.
	    \item[Case 1c:] $|s''|=|V|+2$.
	    \item[Case 1d:] $|s''|=|V|+3$.
	\end{enumerate}
	
	\noindent{\bf Case 1a: $|s''|=1$} 

\begin{tikzpicture}
\path[use as bounding box] (-5,1.5) rectangle (5,-2);
\draw (-4,0) -- (4,0); 
\draw (-4,0.6) -- (4,0.6); 
\draw (-4,-0.6) -- (4,-0.6); 

\draw (-2.5,0.6) -- (-2.5,-0.6);
\draw (-2.65,-0.4) node {2};
\draw (-2.2,0) -- (-2.2,-0.6);
\draw (-2.35,-0.4) node {2};
\draw (-1.9,0) -- (-1.9,-0.6);
\draw (-2.05,-0.4) node {1};
\draw (-1.6,0) -- (-1.6,-0.6);
\draw (-1.75,-0.4) node {2};
\draw (-1.3,0) -- (-1.3,-0.6);
\draw (-1.45,-0.4) node {3};

\draw (-2.8,0) -- (-2.8,-0.6);
\draw[dashed] (-2.2,-1.2) -- (-2.2,-0.6);

\draw (-3.3,-.9) node {$s$};
\draw (1.5,-.9) node {$w$};

\draw (2.5,0) -- (2.5,0.6);

\draw (-0.3,0) -- (-0.3,0.6);
\draw (-1.25,0.3) node {$V$};
\draw (-0.15,0.3) node {$x$};
\draw (1.25,0.3) node {$V$};

\draw (0,0) -- (0,0.6);
\draw (0.15,0.27) node {$y$};
\draw (0.3,0) -- (0.3,0.6);

\end{tikzpicture}

We have  $|w'|=|VxyV|-1$. Note that $2$ is the last letter of $w$. Therefore, if $|w'|<|w|$, then $VxyV=2w'$ is a factor of $[w]$ which is impossible. Thus $w'=w$, and $VxyV=2w$. However, $VxyV$ and $w$ both label closed walks, and therefore have even length, so that $2w$ has odd length. This is a contradiction.	
	
\noindent{\bf Case 1b: $|s''|=|V|+1$.}

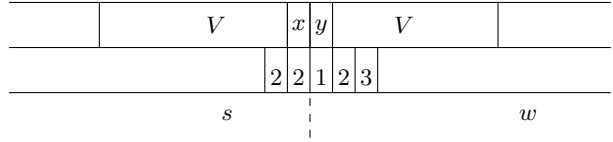
\begin{tikzpicture}
\path[use as bounding box] (-5,1.5) rectangle (5,-2);
\draw (-4,0) -- (4,0); 
\draw (-4,0.6) -- (4,0.6); 
\draw (-4,-0.6) -- (4,-0.6); 

\draw (-0.6,0.) -- (-0.6,-0.6);
\draw (-0.3,0.6) -- (-0.3,-0.6);
\draw (-0.45,-0.4) node {2};
\draw (0,0.6) -- (0,-0.6);
\draw (-0.15,-0.4) node {2};
\draw (0.3,0) -- (0.3,-0.6);
\draw (0.15,-0.4) node {1};
\draw (0.6,0) -- (0.6,-0.6);
\draw (0.45,-0.4) node {2};
\draw (0.9,0) -- (0.9,-0.6);
\draw (0.75,-0.4) node {3};

\draw (-2.8,0) -- (-2.8,0.6);
\draw[dashed] (0,-1.2) -- (0,-0.6);

\draw (-1.1,-.9) node {$s$};
\draw (2.9,-.9) node {$w$};

\draw (2.5,0) -- (2.5,0.6);

\draw (-0.3,0) -- (-0.3,0.6);
\draw (-1.25,0.3) node {$V$};
\draw (-0.15,0.3) node {$x$};
\draw (1.25,0.3) node {$V$};

\draw (0,0) -- (0,0.6);
\draw (0.15,0.27) node {$y$};
\draw (0.3,0) -- (0.3,0.6);

\end{tikzpicture}

In this case, $s''=Vx$, $w'=yV$. Since, $Vxy$ labels a closed walk, $|V|$ must be even. Since $Vx=s''$ has suffix 22, the last letter of $V$ must be 2, and also $x=2$.

If $|V|=2$, then $|w'|=3$. The length $3$ prefix of $w$ is $123$, so that $V=23$, contradicting the fact that the last letter of $V$ must be 2. 

If $|V|=4$, then $|w'|=|yV|=5$, so that $w'$ is the length 5 prefix of $w$, which is 12312. This gives
$y=1$, $V=2312$. Then
 $Vxy=231221$. However, by Lemma~\ref{omega}, 231221 is not a closed walk. 
 
 If $|V|\ge  6$, then $yV=w'$ is of the form $123123h(u)p$, where $u\in A^*$ and $p$ is a prefix of a block, so that $V=23123h(u)p$. Since $x=2$, $p\ne\epsilon$. Then $V^-$ is of the form $23123h(u)p^-$, and is a suffix of $T$, contradicting Condition~\ref{suffix}.
	
\noindent{\bf Case 1c: $|s''|=|V|+2$.}

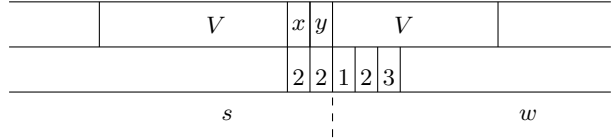
\begin{tikzpicture}
\path[use as bounding box] (-5,1.5) rectangle (5,-2);
\draw (-4,0) -- (4,0); 
\draw (-4,0.6) -- (4,0.6); 
\draw (-4,-0.6) -- (4,-0.6); 

\draw (-0.3,0) -- (-0.3,-0.6);
\draw (0,0.6) -- (0,-0.6);
\draw (-0.15,-0.4) node {2};
\draw (0.3,0) -- (0.3,-0.6);
\draw (0.15,-0.4) node {2};
\draw (0.6,0) -- (0.6,-0.6);
\draw (0.45,-0.4) node {1};
\draw (0.9,0) -- (0.9,-0.6);
\draw (0.75,-0.4) node {2};
\draw (1.2,0) -- (1.2,-0.6);
\draw (1.05,-0.4) node {3};

\draw (-2.8,0) -- (-2.8,0.6);
\draw[dashed] (0.3,-1.2) -- (0.3,-0.6);

\draw (-1.1,-.9) node {$s$};
\draw (2.9,-.9) node {$w$};

\draw (2.5,0) -- (2.5,0.6);

\draw (-0.3,0) -- (-0.3,0.6);
\draw (-1.25,0.3) node {$V$};
\draw (-0.15,0.3) node {$x$};
\draw (1.25,0.3) node {$V$};

\draw (0,0) -- (0,0.6);
\draw (0.15,0.27) node {$y$};
\draw (0.3,0) -- (0.3,0.6);

\end{tikzpicture}

In this case, $s''=Vxy$, $w'=V$. Thus  $xy=22$.
	Also, since $s=33T22$, $V$ is a suffix of $T$, and so must end with the letter 1, by Condition~\ref{1's}. Also $V$ is a prefix of $w$, which commences 123123. We conclude that  $|V|\neq 2$, because the length 2 prefix of $w$ is 12, which does not end in 1. Additionally, $|V|\neq 4$, as this would imply $V=1231$, and $Vxy=123122$ is not a closed walk by Lemma \ref{omega}. If $|V|\ge 6$, then $V=w'=123123h(u)p$ for some $u\in A^*$ and for some prefix of a block $p$. Then $123123h(u)p$ is a suffix of $T$, contradicting Condition~\ref{suffix} on $s$.
	
\noindent{\bf Case 1d: $|s''|=|V|+3$.}

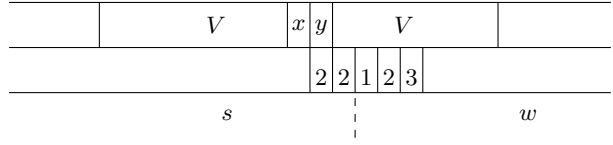
\begin{tikzpicture}
\path[use as bounding box] (-5,1.5) rectangle (5,-2);
\draw (-4,0) -- (4,0); 
\draw (-4,0.6) -- (4,0.6); 
\draw (-4,-0.6) -- (4,-0.6); 

\draw (0,0) -- (0,-0.6);
\draw (0.3,0) -- (0.3,-0.6);
\draw (0.15,-0.4) node {2};
\draw (0.6,0) -- (0.6,-0.6);
\draw (0.45,-0.4) node {2};
\draw (0.9,0) -- (0.9,-0.6);
\draw (0.75,-0.4) node {1};
\draw (1.2,0) -- (1.2,-0.6);
\draw (1.05,-0.4) node {2};
\draw (1.5,0) -- (1.5,-0.6);
\draw (1.35,-0.4) node {3};
\draw (-2.8,0) -- (-2.8,0.6);
\draw[dashed] (0.6,-1.2) -- (0.6,-0.6);

\draw (-1.1,-.9) node {$s$};
\draw (2.9,-.9) node {$w$};

\draw (2.5,0) -- (2.5,0.6);

\draw (-0.3,0) -- (-0.3,0.6);
\draw (-1.25,0.3) node {$V$};
\draw (-0.15,0.3) node {$x$};
\draw (1.25,0.3) node {$V$};

\draw (0,0) -- (0,0.6);
\draw (0.15,0.27) node {$y$};
\draw (0.3,0) -- (0.3,0.6);

\end{tikzpicture}

In this case, $s''=Vxy2$, $2w'=V$. Thus  $xy=12$, since $s=33T22$, and $T$ ends in a 1.
Also $V$ is a prefix of $2w$, which commences 2123123.	
	 If $|V|=2$, then $w'=1$. Then $V=21$, implying that $T$ has $211$ as a suffix, contradicting Condition~\ref{length 2}. Suppose instead that $|V|=4$, so that $V=2123$. Therefore, $Vxy=212312$, which is not a closed walk by Lemma \ref{omega}. If $|V|= 6$, then $V=212312$, and $Vxy=21231212$, which is not a closed walk. If $|V|\ge 8$, then $V=2123123h(u)p$ for some $u\in A^*$ and for some non-empty prefix of a block $p$. Now $s''=Vxy2=V122$, implies that $Vx=2123123h(u)p1$ is a suffix of $T$, contradicting Condition~\ref{suffix} on $s$.
	
\subsubsection*{Case 2: $VxyV=s''ws'$, where $s''$ is a non-empty suffix of $s$, and $s'$ is a non-empty prefix of $s$.}  
	
Because of the restriction mentioned in Remark \ref{repetition_restriction} the possible subcases  are: 
	\begin{enumerate}[wide, labelwidth=!, labelindent=0pt]
	    \item[Case 2a:] 	    $|s'|=|s''|=1$. 
	    \item[Case 2b:] 	    $|s''|=1$, and $|s'|=|V|+1$.
	    \item[Case 2c:]      $|s''|=1$, and $|s'|=|V|+2$.
	    \item[Case 2d:] 	    $|s''|=1$, and $|s'|=|V|+3$.
	    \item[Case 2e:] 	    $|s''|=|V|+1$, and $|s'|=1$.
	    \item[Case 2f:] 	    $|s''|=|V|+1$, and $|s'|=1$.
	    \item[Case 2g:] 	    $|s''|=|V|+1$, and $|s'|=1$. 
	\end{enumerate}

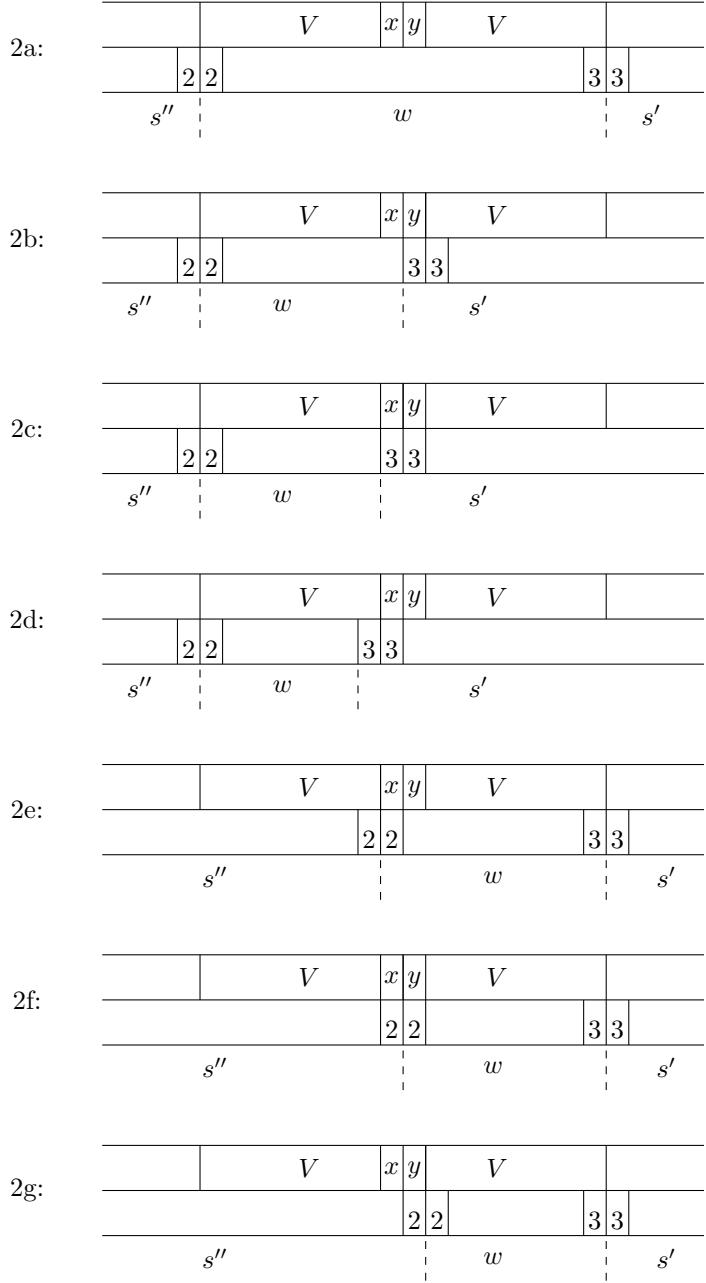
\begin{figure}
\begin{tikzpicture}
\path[use as bounding box] (-5,1) rectangle (5,-1.5);
\draw (-4,0) -- (4,0); 
\draw (-4,0.6) -- (4,0.6); 
\draw (-4,-0.6) -- (4,-0.6); 

\draw (-2.4,0) -- (-2.4,-0.6);
\draw (-2.55,-0.4) node {2};
\draw (-2.7,0) -- (-2.7,-0.6);
\draw (-2.85,-0.4) node {2};
\draw (-3,0) -- (-3,-0.6);

\draw (2.4,0) -- (2.4,-0.6);
\draw (2.55,-0.4) node {3};
\draw (2.7,0.6) -- (2.7,-0.6);
\draw (2.85,-0.4) node {3};
\draw (3,0) -- (3,-0.6);

\draw (-3.2,-.9) node {$s''$};
\draw[dashed] (-2.7,-1.2) -- (-2.7,-0.6);

\draw (0,-.9) node {$w$};
\draw[dashed] (2.7,-1.2) -- (2.7,-0.6);

\draw (3.3,-.9) node {$s'$};

\draw (-2.7,0) -- (-2.7,0.6);

\draw (-0.3,0) -- (-0.3,0.6);
\draw (-1.25,0.3) node {$V$};
\draw (-0.15,0.3) node {$x$};
\draw (1.25,0.3) node {$V$};

\draw (0,0) -- (0,0.6);
\draw (0.15,0.27) node {$y$};
\draw (0.3,0) -- (0.3,0.6);

\draw (-5,0) node {2a:};

\end{tikzpicture}
\begin{tikzpicture}
\path[use as bounding box] (-5,1) rectangle (5,-1.5);
\draw (-4,0) -- (4,0); 
\draw (-4,0.6) -- (4,0.6); 
\draw (-4,-0.6) -- (4,-0.6); 

\draw (-2.4,0) -- (-2.4,-0.6);
\draw (-2.55,-0.4) node {2};
\draw (-2.7,0) -- (-2.7,-0.6);
\draw (-2.85,-0.4) node {2};
\draw (-3,0) -- (-3,-0.6);

\draw (0,0) -- (0,-0.6);
\draw (.15,-0.4) node {3};
\draw (0.3,0.6) -- (0.3,-0.6);
\draw (.45,-0.4) node {3};
\draw (.6,0) -- (.6,-0.6);

\draw (-3.5,-.9) node {$s''$};
\draw[dashed] (-2.7,-1.2) -- (-2.7,-0.6);

\draw (-1.6,-.9) node {$w$};
\draw[dashed] (0,-1.2) -- (0,-0.6);

\draw (1,-.9) node {$s'$};

\draw (-2.7,0) -- (-2.7,0.6);

\draw (-0.3,0) -- (-0.3,0.6);
\draw (-1.25,0.3) node {$V$};
\draw (-0.15,0.3) node {$x$};
\draw (1.25,0.3) node {$V$};
\draw (2.7,0.6) -- (2.7,0);

\draw (0,0) -- (0,0.6);
\draw (0.15,0.27) node {$y$};
\draw (0.3,0) -- (0.3,0.6);
\draw (-5,0) node {2b:};

\end{tikzpicture}

\begin{tikzpicture}
\path[use as bounding box] (-5,1) rectangle (5,-1.5);
\draw (-4,0) -- (4,0); 
\draw (-4,0.6) -- (4,0.6); 
\draw (-4,-0.6) -- (4,-0.6); 

\draw (2.7,0.6) -- (2.7,0);

\draw (-2.4,0) -- (-2.4,-0.6);
\draw (-2.55,-0.4) node {2};
\draw (-2.7,0) -- (-2.7,-0.6);
\draw (-2.85,-0.4) node {2};
\draw (-3,0) -- (-3,-0.6);

\draw (-0.3,0) -- (-0.3,-0.6);
\draw (-.15,-0.4) node {3};
\draw (0,0.6) -- (0,-0.6);
\draw (.15,-0.4) node {3};
\draw (.3,0) -- (.3,-0.6);

\draw (-3.5,-.9) node {$s''$};
\draw[dashed] (-2.7,-1.2) -- (-2.7,-0.6);

\draw (-1.6,-.9) node {$w$};

\draw[dashed] (-0.3,-1.2) -- (-0.3,-0.6);
\draw (1,-.9) node {$s'$};

\draw (-2.7,0) -- (-2.7,0.6);

\draw (-0.3,0) -- (-0.3,0.6);
\draw (-1.25,0.3) node {$V$};
\draw (-0.15,0.3) node {$x$};
\draw (1.25,0.3) node {$V$};

\draw (0,0) -- (0,0.6);
\draw (0.15,0.27) node {$y$};
\draw (0.3,0) -- (0.3,0.6);
\draw (-5,0) node {2c:};

\end{tikzpicture}

\begin{tikzpicture}
\path[use as bounding box] (-5,1) rectangle (5,-1.5);
\draw (-4,0) -- (4,0); 
\draw (-4,0.6) -- (4,0.6); 
\draw (-4,-0.6) -- (4,-0.6); 

\draw (2.7,0.6) -- (2.7,0);

\draw (-2.4,0) -- (-2.4,-0.6);
\draw (-2.55,-0.4) node {2};
\draw (-2.7,0) -- (-2.7,-0.6);
\draw (-2.85,-0.4) node {2};
\draw (-3,0) -- (-3,-0.6);

\draw (-0.6,0) -- (-0.6,-0.6);
\draw (-.45,-0.4) node {3};
\draw (-0.3,0.6) -- (-0.3,-0.6);
\draw (-.15,-0.4) node {3};
\draw (0,0) -- (0,-0.6);

\draw (-3.5,-.9) node {$s''$};
\draw[dashed] (-2.7,-1.2) -- (-2.7,-0.6);

\draw (-1.6,-.9) node {$w$};

\draw[dashed] (-0.6,-1.2) -- (-0.6,-0.6);
\draw (1,-.9) node {$s'$};

\draw (-2.7,0) -- (-2.7,0.6);

\draw (-0.3,0) -- (-0.3,0.6);
\draw (-1.25,0.3) node {$V$};
\draw (-0.15,0.3) node {$x$};
\draw (1.25,0.3) node {$V$};

\draw (0,0) -- (0,0.6);
\draw (0.15,0.27) node {$y$};
\draw (0.3,0) -- (0.3,0.6);
\draw (-5,0) node {2d:};

\end{tikzpicture}

\begin{tikzpicture}
\path[use as bounding box] (-5,1) rectangle (5,-1.5);
\draw (-4,0) -- (4,0); 
\draw (-4,0.6) -- (4,0.6); 
\draw (-4,-0.6) -- (4,-0.6); 

\draw (2.4,0) -- (2.4,-0.6);
\draw (2.55,-0.4) node {3};
\draw (2.7,0.6) -- (2.7,-0.6);
\draw (2.85,-0.4) node {3};
\draw (3,0) -- (3,-0.6);

\draw (-0.6,0) -- (-0.6,-0.6);
\draw (-.45,-0.4) node {2};
\draw (-0.3,0.6) -- (-0.3,-0.6);
\draw (-.15,-0.4) node {2};
\draw (0,0) -- (0,-0.6);

\draw (-2.5,-.9) node {$s''$};
\draw[dashed] (-0.3,-1.2) -- (-0.3,-0.6);

\draw (1.2,-.9) node {$w$};

\draw[dashed] (2.7,-1.2) -- (2.7,-0.6);
\draw (3.5,-.9) node {$s'$};

\draw (-2.7,0) -- (-2.7,0.6);
/
\draw (-0.3,0) -- (-0.3,0.6);
\draw (-1.25,0.3) node {$V$};
\draw (-0.15,0.3) node {$x$};
\draw (1.25,0.3) node {$V$};

\draw (0,0) -- (0,0.6);
\draw (0.15,0.27) node {$y$};
\draw (0.3,0) -- (0.3,0.6);
\draw (-5,0) node {2e:};
\end{tikzpicture}

\begin{tikzpicture}
\path[use as bounding box] (-5,1) rectangle (5,-1.5);
\draw (-4,0) -- (4,0); 
\draw (-4,0.6) -- (4,0.6); 
\draw (-4,-0.6) -- (4,-0.6); 

\draw (2.4,0) -- (2.4,-0.6);
\draw (2.55,-0.4) node {3};
\draw (2.7,0.6) -- (2.7,-0.6);
\draw (2.85,-0.4) node {3};
\draw (3,0) -- (3,-0.6);

\draw (-0.3,0) -- (-0.3,-0.6);
\draw (-.15,-0.4) node {2};
\draw (0,0.6) -- (0,-0.6);
\draw (.15,-0.4) node {2};
\draw (0.3,0) -- (0.3,-0.6);

\draw (-2.5,-.9) node {$s''$};
\draw[dashed] (0,-1.2) -- (0,-0.6);

\draw (1.2,-.9) node {$w$};

\draw[dashed] (2.7,-1.2) -- (2.7,-0.6);
\draw (3.5,-.9) node {$s'$};

\draw (-2.7,0) -- (-2.7,0.6);

\draw (-0.3,0) -- (-0.3,0.6);
\draw (-1.25,0.3) node {$V$};
\draw (-0.15,0.3) node {$x$};
\draw (1.25,0.3) node {$V$};

\draw (0,0) -- (0,0.6);
\draw (0.15,0.27) node {$y$};
\draw (0.3,0) -- (0.3,0.6);
\draw (-5,0) node {2f:};
\end{tikzpicture}

\begin{tikzpicture}
\path[use as bounding box] (-5,1) rectangle (5,-1.5);
\draw (-4,0) -- (4,0); 
\draw (-4,0.6) -- (4,0.6); 
\draw (-4,-0.6) -- (4,-0.6); 

\draw (2.4,0) -- (2.4,-0.6);
\draw (2.55,-0.4) node {3};
\draw (2.7,0.6) -- (2.7,-0.6);
\draw (2.85,-0.4) node {3};
\draw (3,0) -- (3,-0.6);

\draw (0,0) -- (0,-0.6);
\draw (.15,-0.4) node {2};
\draw (0.3,0.6) -- (0.3,-0.6);
\draw (.45,-0.4) node {2};
\draw (0.6,0) -- (0.6,-0.6);

\draw (-2.5,-.9) node {$s''$};
\draw[dashed] (0.3,-1.2) -- (0.3,-0.6);

\draw (1.2,-.9) node {$w$};

\draw[dashed] (2.7,-1.2) -- (2.7,-0.6);
\draw (3.5,-.9) node {$s'$};

\draw (-2.7,0) -- (-2.7,0.6);
/
\draw (-0.3,0) -- (-0.3,0.6);
\draw (-1.25,0.3) node {$V$};
\draw (-0.15,0.3) node {$x$};
\draw (1.25,0.3) node {$V$};

\draw (0,0) -- (0,0.6);
\draw (0.15,0.27) node {$y$};
\draw (0.3,0) -- (0.3,0.6);
\draw (-5,0) node {2g:};
\end{tikzpicture}

\caption{Subcases of Case 2}\label{Case 2}\end{figure}

\noindent{\bf Case 2a}

Here $VxyV=2w3$. Note that $w=h(v)$, so $|w|\ge  12$. Therefore $|VxyV|\ge 14$, and $|V|\ge  6$. Since $v$ ends in $b$, $w$ ends in 132132, so that 1323 is a suffix of $VxyV$, hence of $V$. This forces 1323 to appear in the prefix $V$ of $VxyV$, hence in $w$. However, $1323$ does not appear as a factor of $w=h(v)$, giving a contradiction.
	
\noindent{\bf Cases 2b, 2c, 2e, 2f}

Case 2b is impossible; we have  $|s''|=1$, and $|s'|=|V|+1$.
This implies that $V$ begins with 2, but also with 3. 

Similarly, Case 2c requires that $V$ begins with 2, and also that $V$ begins with 1, by Condition~\ref{1's} on $s$. 

Case 2e implies that $V$ ends with both 2 and 3. 

Case 2f implies that $V$ ends with both 1 and 3.
	
\noindent{\bf Cases 2d, 2g}
	These cases both imply that $w$ is a factor of $s$. Because $w=h(v)$, and $|v|\ge  2$, this contradicts Condition~\ref{no h} on $s$.
		
\subsubsection*{Case 3: $VxyV=w''s'$, where $w''$ is a non-empty suffix of $w$, and $s'$ is a non-empty prefix of $s$.} This contradicts Condition~\ref{Vxy}. We conclude that $|w''|\ge 2$, so that 32 is a suffix of $w''$. 
Because of the restriction mentioned in Remark \ref{repetition_restriction} the possible subcases are given below.

	\begin{enumerate}[wide, labelwidth=!, labelindent=0pt]
	    \item[Case 3a:] 	    $|s'|=1$. 
	    \item[Case 3b:] 	    $|s'|=|V|+1$.
	    \item[Case 3c:]      $|s'|=|V|+2$.
	    \item[Case 3d:] 	    $|s'|=|V|+3$.
	\end{enumerate}

\begin{figure}

\begin{tikzpicture}
\path[use as bounding box] (-5,1) rectangle (5,-1.5);
\draw (-4,0) -- (4,0); 
\draw (-4,0.6) -- (4,0.6); 
\draw (-4,-0.6) -- (4,-0.6); 

\draw (2.4,0) -- (2.4,-0.6);
\draw (2.55,-0.4) node {3};
\draw (2.7,0.6) -- (2.7,-0.6);
\draw (2.85,-0.4) node {3};
\draw (3,0) -- (3,-0.6);
0.6);

\draw (1.2,-.9) node {$w''$};

\draw[dashed] (2.7,-1.2) -- (2.7,-0.6);
\draw (3.5,-.9) node {$s'$};

\draw (-2.7,0) -- (-2.7,0.6);

\draw (-0.3,0) -- (-0.3,0.6);
\draw (-1.25,0.3) node {$V$};
\draw (-0.15,0.3) node {$x$};
\draw (1.25,0.3) node {$V$};

\draw (0,0) -- (0,0.6);
\draw (0.15,0.27) node {$y$};
\draw (0.3,0) -- (0.3,0.6);
\draw (-5,0) node {3a:};
\end{tikzpicture}

\begin{tikzpicture}
\path[use as bounding box] (-5,1) rectangle (5,-1.5);
\draw (-4,0) -- (4,0); 
\draw (-4,0.6) -- (4,0.6); 
\draw (-4,-0.6) -- (4,-0.6); 

\draw (0,0) -- (0,-0.6);
\draw (.15,-0.4) node {3};
\draw (0.3,0.6) -- (0.3,-0.6);
\draw (.45,-0.4) node {3};
\draw (.6,0) -- (.6,-0.6);

\draw (-1.6,-.9) node {$w''$};
\draw[dashed] (0,-1.2) -- (0,-0.6);

\draw (1,-.9) node {$s'$};

\draw (-2.7,0) -- (-2.7,0.6);

\draw (-0.3,0) -- (-0.3,0.6);
\draw (-1.25,0.3) node {$V$};
\draw (-0.15,0.3) node {$x$};
\draw (1.25,0.3) node {$V$};

\draw (0,0) -- (0,0.6);
\draw (0.15,0.27) node {$y$};
\draw (0.3,0) -- (0.3,0.6);
\draw (-5,0) node {3b:};

\end{tikzpicture}

\begin{tikzpicture}
\path[use as bounding box] (-5,1) rectangle (5,-1.5);
\draw (-4,0) -- (4,0); 
\draw (-4,0.6) -- (4,0.6); 
\draw (-4,-0.6) -- (4,-0.6); 
\draw (2.7,0.6) -- (2.7,0);

\draw (-0.3,0) -- (-0.3,-0.6);
\draw (-.15,-0.4) node {3};
\draw (0,0.6) -- (0,-0.6);
\draw (.15,-0.4) node {3};
\draw (.3,0) -- (.3,-0.6);

\draw (-1.6,-.9) node {$w''$};

\draw[dashed] (-0.3,-1.2) -- (-0.3,-0.6);
\draw (1,-.9) node {$s'$};

\draw (-2.7,0) -- (-2.7,0.6);

\draw (-0.3,0) -- (-0.3,0.6);
\draw (-1.25,0.3) node {$V$};
\draw (-0.15,0.3) node {$x$};
\draw (1.25,0.3) node {$V$};

\draw (0,0) -- (0,0.6);
\draw (0.15,0.27) node {$y$};
\draw (0.3,0) -- (0.3,0.6);
\draw (-5,0) node {3c:};

\end{tikzpicture}

\begin{tikzpicture}
\path[use as bounding box] (-5,1) rectangle (5,-1.5);
\draw (-4,0) -- (4,0); 
\draw (-4,0.6) -- (4,0.6); 
\draw (-4,-0.6) -- (4,-0.6); 

\draw (2.7,0.6) -- (2.7,0);

\draw (-0.6,0) -- (-0.6,-0.6);
\draw (-.45,-0.4) node {3};
\draw (-0.3,0.6) -- (-0.3,-0.6);
\draw (-.15,-0.4) node {3};
\draw (0,0) -- (0,-0.6);

\draw (-1.6,-.9) node {$w$};

\draw[dashed] (-0.6,-1.2) -- (-0.6,-0.6);
\draw (1,-.9) node {$s'$};

\draw (-2.7,0) -- (-2.7,0.6);
/
\draw (-0.3,0) -- (-0.3,0.6);
\draw (-1.25,0.3) node {$V$};
\draw (-0.15,0.3) node {$x$};
\draw (1.25,0.3) node {$V$};

\draw (0,0) -- (0,0.6);
\draw (0.15,0.27) node {$y$};
\draw (0.3,0) -- (0.3,0.6);
\draw (-5,0) node {3d:};

\end{tikzpicture}

\caption{Subcases of Case 3}\label{Case 3}\end{figure}
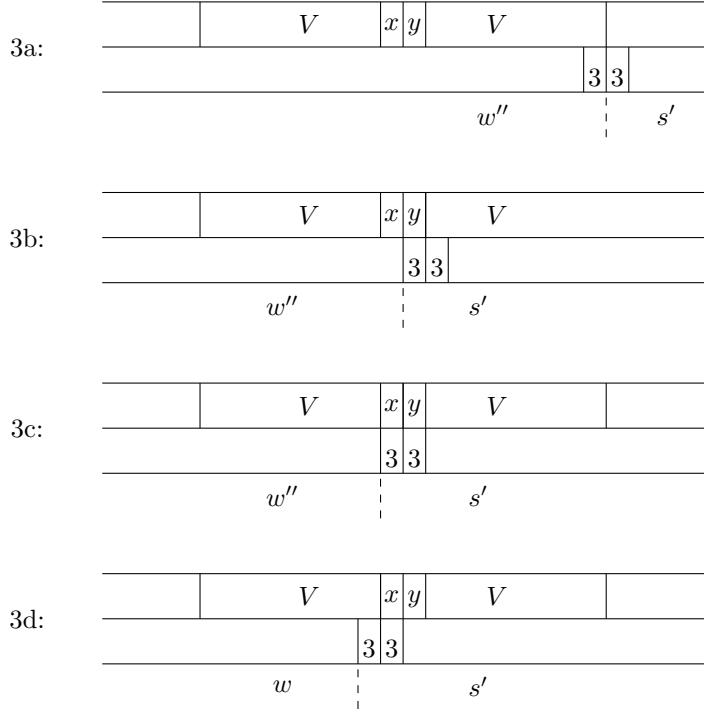

\subsubsection*{Case 3a: $|s'|=1$.} In this case, $s'=3$, $VxyV=w''3$. Since $|w''|\ge 2$, and $|VxyV|$ is even, we conclude that $|w''|\ge 3$. Since $w=h(v)$ has 132132 as a suffix, $w''$ has 132 as a suffix, and $VxyV$ ends in 1323. 

If $|V|=2$, then $VxyV=231323$, and $w$ has 23132 as a suffix, a contradiction. 

If $|V|\ge 4$, then 1323 is a suffix of $V$, which is a prefix of suffix $w''$ of $w$. However, $323$ does not appear in $w$, so Case 3a ends in contradiction.

\subsubsection*{Case 3b: $|s'|=|V|+1$}
	
Since $VxyV=w''s'$, $|w''|=|V|+1=|s'|$. The length 6 suffix of $w=h(v)$ is $132132$, and 331 is a prefix of $s$. Recall that $|V|$ is even. 

If $|V|=2$, then $VxyV$ is 132331, which is not a closed walk by Lemma~\ref{omega}. 

If $|V|=4$, then $Vx$ is the suffix 32132 of $w$. However, $yV=s'$, which has prefix 331, so that $V$ begins with 31, contradicting $Vx=32132$. 

If $|V|\ge  6$, then $V2$ is a suffix of $w$, so that $V=qh(u)13213$, where $q$ is the non-empty suffix of some block, and $u\in A^*$. Let $\hat{q}$ be obtained from $q$ by deleting its first letter. Then $\hat{q}h(u)13213$ is a prefix of $T$, contradicting Condition~\ref{prefix} on $s$.
	
\subsubsection*{Case 3c: $|s'|=|V|+2$.}
 In this case, $V=w''$. Also, the first letter of $V$ must be 1, the third letter of $s$. However, the length 2 suffix of $w$ is 32, and the length 4 suffix of $w$ is 2132. Neither of these begin with $1$, so $|V|\ge 6$. Therefore, because $s'=xyw''$, $T$ has $qh(u)13213$ as a prefix, where $q$ is the suffix of some block, and $u\in A^*$. This contradicts Condition~\ref{prefix} of $s$.
	
\subsubsection*{Case 3d: $|s'|=|V|+3$.}
 In this case, $w''3=V$. Since  331 is a prefix of $s$, $s'$ begins 331, and $xy=31$. In fact, $s'=3xyV=331w''3$.
 
 If $|V|=2$, then $|w''|=1$, so that $w''=2$ and $Vxy=2331$, However, this is not a closed walk by Lemma \ref{omega}. 
 
 If $|V|=4$, then $w''=132$, so that $Vxy=132331$, which is not a closed walk. 
 
 If $|V|=6$, then $w''=32132$, and $s'=331321323$, and $T$ has $13213$ as a prefix, contradicting Condition~\ref{prefix} on $s$.
 
 If $|V|\ge 8$, then $w'=qh(u)132132$, where $q$ is a non-empty suffix of a block, and $u\in A^*$. Then, $s'=331w'3$, and $T$ has as a prefix $1qh(u)13213$, contradicting Condition~\ref{prefix} on $s$.
	
\subsection*{Case 4: $VxyV=w''sw'$, where $w''$ is a non-empty suffix of $w$, and $w'$ is a non-empty prefix of $w$.} 
	
In this case, either the length 2 prefix or the length 2 suffix of $s$ is a factor of $V$, since only one of these can overlap $xy$. This gives a contradiction, as mentioned in Remark \ref{repetition_restriction}.
\end{proof}

\section{Computer search for the words $s$}

Table \ref{stable} gives a list of words $s$ that fulfill the conditions of Theorem~\ref{linking words}, found by computer search, along with the lengths of the images of these words under $f$. 

\small
\begin{table}
	\centering
	\begin{multicols}{2}
		\begin{tabular}{ |c|c| }
			\hline
		$s$	&$|f(s)|$\\ 
			\hline
331313123231212122&54\\
331313232321212122&55\\
331232323231212122&56\\
33131312121231212122&57\\
331323232321313122&58\\
33123232132121212122&59\\
33123123213231212122&60\\
33131323231231212122&61\\
33132323123231212122&62\\
33131323232312132122&63\\
33132323231232132122&64\\
33132323231313232122&65\\
3312321313123231212122&66\\
3312323232313121212122&67\\
3312323213232321212122&68\\
3312312323231232132122&69\\
3312323232313132132122&70\\
3312323232312132323122&71\\
331212323123213231212122&72\\
331212323131323231212122&73\\
331212323123213232312122&74\\
331212323232132321313122&75\\
331212323213132323232122&76\\
331212323232132323232122&77\\
331232313132323232313122&78\\
331313232323123232323122&79\\
33121212323232321323212122&80\\
			\hline
		\end{tabular}
		
		\begin{tabular}{ |c|c| }
			\hline
			$s$&$|f(s)|$\\ 
			\hline
33121212313132323232313122&81\\
33121213232323231313232122&82\\
33121213232323123232323122&83\\
33123123232321323232313122&84\\
33123123232323213232323122&85\\
3312121212323232312132323122&86\\
3312121232323232132321313122&87\\
3312121232323213132323232122&88\\
3312121232323232132323232122&89\\
3312123232313132323232313122&90\\
3312132323232132323232313122&91\\
3312323213232323123232323122&92\\
331212121232313132323232313122&93\\
331212121313232323123232323122&94\\
331212123232323213232132323122&95\\
331212132323232313123232323122&96\\
331212323123232323213232323122&97\\
331231232323232132323232313122&98\\
33121212123123232321323232313122&99\\
33121212123123232323213232323122&100\\
33121212312323232321323232313122&101\\
33121212323232313132323232313122&102\\
33121232312323232313132323232122&103\\
33121232323213232323123232323122&104\\
33123232132323232313123232323122&105\\
33123232323123232323213232323122&106\\
3312121212323213232323123232323122&107\\
			\hline
		\end{tabular}
	\end{multicols}
	\caption{Values for $s$}
	\label{stable}
\end{table}
\normalsize

Let $s$ be a word from Table \ref{stable}. One checks that $\Delta(f(s))$ ends in $ab$. Let $\psi=	\Delta(f(s))^{--}$, $\omega=\Delta(f(w))^{--}$. 
One checks that $\psi$ is level.
The circular word encoded by $f(sw)$ is
\begin{eqnarray*}
[\Delta(f(sw))^{--}]&=&[\Delta(f(s)f(w))^{--}]\\
	&=&\Delta(f(s))^{--}\Delta(f(w))^{--}\mbox{as per Remark~\ref{concatenation}}\\
	&=&\psi\omega. 
	\end{eqnarray*}
By Theorem~\ref{linking words}, $\psi\omega$ is a circular square-free word.
	
	Recall that $|\omega|_a=|\omega|_b=|\omega|_c$ by Corollary \ref{basic_squarefree_level}. Let $\alpha,\beta\in A$. We have
$$\begin{array}{ccccc}
|\psi|_\alpha-1&\le& |\psi|_\beta&\le& |\psi|_\alpha+1\\
	|\psi|_\alpha+|\omega|_\alpha-1&\le& |\psi|_\beta+|\omega|_\alpha&\le& |\psi|_\alpha+|\omega|_\alpha+1\\
	|\psi|_\alpha+|\omega|_\alpha-1&\le& |\psi|_\beta+|\omega|_\beta&\le& |\psi|_\alpha+|\omega|_\alpha+1\\
	|\psi\omega|_\alpha-1&\le& |\psi\omega|_\beta&\le& |\psi\omega|_\alpha+1
	\end{array}$$
	
	implying that $\psi\omega$ is level. 

\begin{corollary}\label{getLonger}
	Let $[v]$ be a circular square-free word over $A$ such that $v$ has prefix $a$ and suffix $b$. Let $w=h(v)$, and let $s$ be a word from Table \ref{stable}. Then $f(ws)$ encodes a level ternary circular square-free word of length $|f(w)|+|f(s)|$. 
\end{corollary}

Note that for any circular square-free word $[v]$ with $|v|\ge 2$, the first and last letters of $v$ are different, so that up to a permutation of the alphabet, we may assume that 
$v$ has prefix $a$ and suffix $b$. We therefore have the following corollary:

\begin{corollary}\label{18m+r}
Let $s$ be a word from Table \ref{stable}. There is a level ternary circular square-free word of length $18m+|f(s)|$, for any positive integer $m\ne 1,5,7,9,10,14,17.$ 
\end{corollary}

 \begin{theorem}\label{90}
Suppose $n$ is an integer, $n\ge 90$. There is  a level ternary circular square-free word of length $n$.
\end{theorem}
\begin{proof}
Note that Table~\ref{stable} gives words $s$ with $|f(s)|$ taking on values from $54=3(18)$ to $107=5(18)+17$, inclusive.

Suppose first that $90\le n\le 143$. Then $54\le n-36\le 107.$ Choose $s$ from Table~\ref{stable} with $|f(s)|= n-36$. The result follows from Corollary~\ref{18m+r} with $m=2$.

If $n\ge 144$, we can write $n$ in the form $n=18m+r$, $54\le r\le 107$, for three consecutive integers $m\ge 2$. Since the set $\{1,5,7,9,10,14,17\}$ does not contain three consecutive integers, a positive integer $m\notin \{1,5,7,9,10,14,17\}$ and $s$ from Table~\ref{stable} with $|f(s)|=r$, such that $n=18m+r$. The result follows by Corollary~\ref{18m+r}.
 \end{proof}
\section{Main theorem}

\begin{thm}[Main Theorem]
There is  a level ternary circular square-free word of length $n$,
for each positive integer $n$, $n\ne 5, 7, 9, 10, 14, 17$. \end{thm}
\begin{proof}
Theorem~\ref{90} shows that there is a level ternary circular square-free word of length $n$,
for each integer $n\ge 90.$
An exhaustive computer search shows that a level ternary square-free word exists for each positive integer $n$, $n\le 89$, besides $n\ne 5, 7, 9, 10, 14, 17$. The circular words $[a]$, $[ab]$ and $[abc]$ give examples for $1\le n\le 3.$ We give encodings of words of the other lengths in Table~\ref{short words}. This establishes our main theorem.
\end{proof}

\small
\begin{table}
	\centering
	\begin{multicols}{2}
		\begin{tabular}{ |c|c| }
			\hline
			$|w|$ & encoding of $w$ \\ 
			\hline
			
			\small

$4$&$3$\\
$6$&$22$\\
$8$&$33$\\
$11$&$3121$\\
$12$&$3212$\\
$13$&$3132$\\
$15$&$121212$\\
$16$&$122122$\\
$18$&$312312$\\
$19$&$313123$\\
$20$&$331232$\\
$21$&$323232$\\
$22$&$13131212$\\
$23$&$13132122$\\
$24$&$32321212$\\
$25$&$33131321$\\
$26$&$31313232$\\
$27$&$1231212122$\\
$28$&$3132121212$\\
$29$&$3312312121$\\
$30$&$3131313212$\\
$31$&$3313131231$\\
$32$&$3232312132$\\
$33$&$131313121212$\\
$34$&$3232321323$\\
$35$&$131323212122$\\
$36$&$123232132122$\\
$37$&$331313232121$\\
$38$&$331232321321$\\
$39$&$13131321212122$\\
$40$&$33131321212121$\\
$41$&$33131313212121$\\
$42$&$33123132121321$\\
$43$&$33123131321321$\\
$44$&$33132323121321$\\
$45$&$1231313131212122$\\
$46$&$33131323232321$\\
$47$&$1313232321212122$\\
$48$&$3313131232312121$\\
$49$&$3313132323212121$\\

			\hline
		\end{tabular}
		
		\begin{tabular}{ |c|c| }
			\hline
			$|w|$& encoding of $w$ \\ 
			\hline

$50$&$3312323232312121$\\
$51$&$123232132121212122$\\
$52$&$3313232323213131$\\
$53$&$331232321321212121$\\
$54$&$331231232132312121$\\
$55$&$331313232312312121$\\
$56$&$331323231232312121$\\
$57$&$331313232323121321$\\
$58$&$331323232312321321$\\
$59$&$331323232313132321$\\
$60$&$12323213232321212122$\\
$61$&$33123232323131212121$\\
$62$&$33123232132323212121$\\
$63$&$33123123232312321321$\\
$64$&$33123232323131321321$\\
$65$&$33123232323121323231$\\
$66$&$3312123231232132312121$\\
$67$&$3312123231313232312121$\\
$68$&$3312123231232132323121$\\
$69$&$1212323232132323232122$\\
$70$&$3312123232131323232321$\\
$71$&$3312123232321323232321$\\
$72$&$3312323131323232323131$\\
$73$&$3313132323231232323231$\\
$74$&$331212123232323213232121$\\
$75$&$331212123131323232323131$\\
$76$&$331212132323232313132321$\\
$77$&$331212132323231232323231$\\
$78$&$331231232323213232323131$\\
$79$&$331231232323232132323231$\\
$80$&$33121212123232323121323231$\\
$81$&$12121232323232132323232122$\\
$82$&$33121212323232131323232321$\\
$83$&$33121212323232321323232321$\\
$84$&$33121232323131323232323131$\\
$85$&$33121323232321323232323131$\\
$86$&$33123232132323231232323231$\\
$87$&$3312121212323131323232323131$\\
$88$&$3312121213132323231232323231$\\
$89$&$3312121232323232132321323231$\\
$90$&$3312121323232323131232323231$\\

			\hline
		\end{tabular}
	\end{multicols}
	\caption{Encodings of short level ternary circular square-free words $w$}
	\label{short words}
\end{table}

\end{document}